\documentclass[12pt,reqno,oneside]{amsart}
\usepackage{amsmath,amsthm,amsfonts,amssymb}
\usepackage[mathscr]{eucal}
\usepackage{indentfirst}
\usepackage{url}

\theoremstyle{plain}
\newtheorem{teo}{Theorem}[section]
\newtheorem{cor}[teo]{Corollary}
\newtheorem{lem}[teo]{Lemma}
\newtheorem{prop}[teo]{Proposition}
\newtheorem*{raabe}{Raabe�s Test}

\theoremstyle{definition}

\newtheorem{exa}[teo]{Example}
\newtheorem{rem}[teo]{Remark}

\numberwithin{equation}{section}

\def\bbR{{\mathbb R}}
\def\bbP{{\mathbb P}}
\def\bbZ{{\mathbb Z}}
\def\bbN{{\mathbb N}}
\def\bbE{{\mathbb E}}
\def\qed{\hfill $\square$}

\def\hofp{\textit{Homogeneous Firework Process}}
\def\horfp{\textit{Reverse Homogeneous Firework Process}}
\def\hefp{\textit{Heterogeneous Firework Process}}
\def\herfp{\textit{Reverse Heterogeneous Firework Process}}

\begin{document}

\baselineskip=26pt

\title{Rumour Processes on $\bbN$}

\author{Valdivino~V.~Junior}
\author{F\'abio~P.~Machado}
\author{Mauricio~Zuluaga}

\address[F\'abio~P.~Machado]
{Institute of Mathematics and Statistics
\\ University of S\~ao Paulo \\ Rua do Mat\~ao 1010, CEP
05508-090, S\~ao Paulo, SP, Brazil.}

\address[Valdivino~V.~Junior]
{Federal University of Goias
\\ Campus Samambaia, CEP 74001-970, Goi\^ania, GO, Brazil.}

\address[Mauricio~Zuluaga]
{Department of Statistics, Federal University of Pernambuco
\\ Cidade Universit\'aria, CEP 50740-540, Recife, PE, Brazil.}

\email{vvjunior@mat.ufg.br, (fmachado@ zuluaga@)ime.usp.br}

\thanks{Research supported by CNPq (306927/2007-1) and FACEPE (0126-1-02/06).}



\date{\today}

\begin{abstract}
We study four discrete time stochastic systems on $\bbN$ modeling
processes of rumour spreading. The involved individuals can either have an active or
a passive role, speaking up or asking for the rumour. The appetite in
spreading or hearing the rumour is represented by a set of random variables whose distributions
may depend on the individuals. Our goal is to understand - based on those random
variables distribution - whether the probability of having an infinite
set of individuals knowing the rumour is positive or not.
\end{abstract}

\maketitle

\section{Introduction}
\label{S: Introduction}

Until a few decades ago, epidemic and rumour models where treated under the
same class of models. While there is a clear similitude
among the status of the individuals in the models (susceptible are ignorants,
immunes are stiflers and infected are spreaders) the rates at which individuals
change their status might be qualitatively different (Pearce~\cite{Pearce}).
Generally speaking, the production of stiflers is definitely more complex than the production of
immune individuals.

Lately the mathematics of rumors has observed a good deal of interest. The
focus used to be at deterministic or stochastic models, modeling
homogeneously mixed populations living on spaces with no structure
as the Maki-Thompson (Maki and Thompson~\cite{MT} and
Sudbury~\cite{Sudbury}) and Daley-Kendall (Daley and Kendal~\cite{DK} and
Pittel~\cite{Pittel}) models.
Among the possible variations one can find in recent literature are competing
rumours (Kostka \textit{et al}~\cite{KOW}), more than two
people meeting at a time (Kesten and Sidoravicius~\cite{Kesten}), moving
agents (Kurtz \textit{et al}~\cite{Kurtz})
and rumours through tree-like graphs (Lebensztayn and Rodriguez~\cite{LR}
and Lebensztayn \textit{et al}~\cite{LMM}),
complex networks (Isham \textit{et al}~\cite{Isham}),
grids (Roy \textit{et al}~\cite{Roy})
and multigraphs (Bertachi and Zucca~\cite{BZ}).

Still, the most important question for both models, epidemic and rumour,
is in terms of a rumour model, if a spreader (an individual
who wants to see the rumour spread) is introduced into a reservoir of ignorants under
what conditions the rumour will spread to a large proportion of the population, instead
of dying out quickly without having done so. Another important question is, if it
does not dies out quickly, what is the final proportion of individuals hit by the rumour?

We study discrete time stochastic systems on $\bbN=\{0,1,2,\dots\}$ which dynamic
is as follows. First, consider that at time zero all vertices of $\bbN$ are declared inactive,
except for the origin, which is active. It instantly exerts influence on its neighbors vertices,
activating a contiguous random set of them placed on its right. In general, that is the behavior
of every vertex in case it is activated.

We take into account an \textit{homogeneous} and an \textit{heterogeneous} versions for what
we call the \textit{radius of influence} of a vertex. In the \textit{homogeneous} version, as a rule, the
next moment to what it has been activated, each active vertex carries the same (random) behavior of the
origin, independent of it and of everything else.
We also deal with an \textit{heterogeneous} version where each vertex, if activated, has a distinct distribution
for its radius of influence.

We say that the process \textit{survives} if the amount of vertices
activated is infinite. Otherwise we say the process \textit{dies out}. We call this the \textit{Firework
Process}, associating the activation dynamic of a vertex to a rumour process. Vertices
become spreaders as soon as they are activated. Next time, they propagate the rumour and
immediately become stiflers.

A possible variation is what we call \textit{Reverse Firework Process}. In this variation
a vertex, instead of being hit by a rumour, defines a set of neighbors on its left to
which it asks once someone in this set hears the rumour.
We call this variation
\textit{Reverse Firework Process}. We also deal with an \textit{homogeneous} and an \textit{heterogeneous}
versions of this variation. The models are shown to be qualitatively different in some pertinent cases.

Our main interest is to establish whether each process has positive probability of survival which is
equivalent to a rumour propagation.
This is done according to the distribution of the random variable that defines
the radius of influence of each active vertex.

The paper is organized as follows. Section 2 presents the main results.
Section 3 brings the proofs for the main results together with
auxiliary lemmas and handy inequalities. In Section 4 we present examples where
some conditions can be verified.

\section{Main Results}
\label{S: MR}

\subsection{Firework Process}
\label{S: FP}

Consider $\{u_i\}_{i \in \bbN} $ a set of vertices of $\bbN$ such that
$0=u_0 < u_1 < u_2 <  \cdots $ and a set of independent
random variables $\{R_i\}_{i \in \bbN}$ assuming values in $\bbR_+$ whose joint
distribution is $\bbP$. The \textit{Firework Process} can be formally defined
in the following way. At time 0, an explosion
of size $R_0$ comes from the origin, activating all vertices $u_i \leq R_0$.
As a rule, at every discrete time $t$ all vertices  $ u_j $ activated at time $t-1$
generate their explosions (whose \textit{radius of influence} is $R_j$), and
they do this just once, activating the vertices $ u_i $ (only those which has
not been activated before) such that $u_j < u_i \leq u_j + R_j.$ Observe that
except for the set of vertex $\{u_i\}$, all others vertices are non-actionable,
meaning that the random variable associated to them is 0 almost surely.

If for all $u_j$ activated
at time $t-1$ there are no vertices $u_i$ in this latter condition the
process \textit{dies out}. That means the rumour reaches only a finite amount of
individuals. If, on the contrary the process never stops, we say it \textit{survives},
meaning that the rumour reaches an infinity number of individuals.
We call the process \textit{homogeneous} if all $R_i$ have the same
distribution and $u_i=i$ for all $i$. Otherwise we call it \textit{heterogeneous}.
We focus to the cases $\bbP(R_i < 1) \in (0,1)$ for all $i$.
%Unless stated differently, we assume $u_i = i$ for all $i$.

Let us consider the following events
\begin{itemize}
\item $V_n=$ \textit{the vertex $u_n$ is hit by an explosion},
\item $V= \lim_{n \rightarrow \infty } V_n$.
\end{itemize}

\subsubsection{The Homogeneous case}
\label{S: hoc}

\begin{teo}
\label{T: Criterio}
For the \hofp, consider
\[a_n=\prod_{i=0}^{n}\bbP(R < i+1).\]

Then
\[ \sum_{n=1}^{\infty}a_n = \infty \hbox{ if and only if } \bbP[V]=0. \]

Besides
\begin{eqnarray}
\label{PVinequalityGEQ}
&\bbP(V)& \geq \prod_{j=0}^{\infty}\Big[1-\prod_{i=0}^{j}\bbP(R<i+1)\Big], \\
\label{PVinequalityLEQ}
&\bbP(V)& \leq 1 - \bbP(R=0) - \sum_{k=1}^{n}\displaystyle
\left[\bbP(R = k) \prod_{j=0}^{k-1}\bbP[R \leq
j]\right].
\end{eqnarray}
\end {teo}

\begin{cor}
\label{C: FPHo}
For the \hofp, consider
\[ L = \lim_{n \rightarrow \infty}n\bbP(R \geq n). \]

We have that
\begin{enumerate}
\item[\textit{(I)}] If $L > 1$ then $\bbP[V]>0.$
\item[\textit{(II)}] If $L < 1$ then $\bbP[V]=0.$
\item[\textit{(III)}] If $L = 1$ and there exists $N$ such that for all $n \geq N$
\begin{displaymath}
\bbP(R \geq n) \leq \frac{1}{n-1}, \textit{ then } \bbP[V]=0.
\end{displaymath}
%\item[\textit{(IV)}] $\bbP(V) \geq
%\prod_{j=0}^{\infty}\Big[1-\prod_{i=0}^{j}\bbP(R < i+1)\Big].$

\end{enumerate}
\end {cor}

\begin{rem}
\label{R: ERfiniteDyes}
Consider a \hofp\ with $R$ assuming values on $\bbN.$
Observe that, in this case, if $\bbE[R] < \infty $ then $L=0$. Consequently
for $R$ assuming values on $\bbN,$

\[ \bbE[R] < \infty \Rightarrow  \bbP[V]=0.\]
\end{rem}

Next result gives a criteria for the case when the distribution
of the random variable $R$ is a power law.

\begin{cor}
\label{C: lei de potencia}
Let $\alpha > 1$ and $Z_{\alpha}$ be an appropriate constant.
Consider the \hofp\ such that
\begin{equation}
\label{E: Power Law}
\bbP(R = k) = \frac{Z_{\alpha}}{(k+1)^{\alpha}} \textrm {
for } k \in \bbN.
\end{equation}

\begin{itemize}
\item[\textit{(I)}] If $\alpha < 2$ then $\bbP[V]>0.$
\item[\textit{(II)}] If $\alpha \geq 2$ then $\bbP[V]=0.$
\end{itemize}
\end{cor}

\begin{rem}
\label{R: PowerLaw}
Observe that for the \hofp\ if $R$ has a power law distribution
as in~(\ref{E: Power Law}), with
$\alpha =2$, we have that

\[ \bbE[R] = \infty \hbox{ and } \bbP[V]=0. \]
\end{rem}

\subsubsection{The Heterogeneous case}

\begin{rem}
\label{R: PFHeMorte} Consider the \hefp. One can get a
sufficient condition for $\bbP[V]=0$ ($\bbP[V]>0$) by a coupling argument.
Consider $\bbP(R_i \geq k ) \leq \mathbf{P}(R \geq k)$
($\bbP(R_i \geq k ) \geq \mathbf{P}(R \geq k)$)
for some random variable $R$ which distribution $\mathbf{P}$ satisfies
$\lim_{n \rightarrow \infty}n \mathbf{P}(R \geq n) < 1$
($\lim_{n \rightarrow \infty}n \mathbf{P}(R \geq n) > 1$).
Finally use part ($II$) (part ($I$)) of Corollary~\ref{C: FPHo}.
\end{rem}

\begin {teo}
\label{T: PFHeVive}
Consider a \hefp\ which actionable vertices are at integer positions
$0=u_0 < u_1 < u_2 < \dots $ such that $u_{n+1} - u_n \leq m$, for $m \geq 1.$
Besides, let us assume $\bbP (R_{n} < m) \in (0,1)$ for all $n.$
\begin{enumerate}
\item[\textit{(I)}] If $\sum_{n=0}^{\infty}[\bbP(R_{n} < tm)]^t < \infty$ for some $t \geq 1$
then $\bbP[V]>0.$
\item[\textit{(II)}] If for some random variable $R,$ which distribution is $\mathbf{P},$
the following conditions hold
\begin{itemize}
\item $\mathbf{P}(R \geq k) - \bbP(R_{n} \geq k) \leq b_k$
for all $ k \geq 0$ and all $n \geq 0,$
\item $\lim_{n \rightarrow \infty}n[\mathbf{P}(R \geq n) - b_n] > m,$
\item $\lim_{n \rightarrow \infty}b_n = 0.$
\end{itemize}
Then $\bbP[V]>0.$
\item[\textit{(III)}] $\bbP(V) \geq
\prod_{j=0}^{\infty}\Big[1-\prod_{i=0}^{j}\bbP(R_{j-i} < (i+1)m)\Big].$
\end{enumerate}
\end {teo}

\subsection{Reverse Firework Process}
\label{S: RFP}

Consider $\{u_i\}_{i \in \bbN}$ a set of vertices of $\bbN$ such that
$0=u_0 < u_1 < u_2 <  \cdots $ and a set of independent
random variables $\{R_i\}_{i \in \bbN}$ assuming values in $\bbN$ which joint
distribution is $\bbP$. The \textit{Reverse Firework Process} can be defined as
follows. At time 0, only the origin is activated. At time 1, explosions of size
$R_i $ towards the origin, come from all vertices of $\{u_i\}_{i \in \bbN}.$
All vertices $u_i \leq R_i$ are activated. As a rule, at discrete times $t$
the set of vertices $ u_j $ which can find a vertex activated at time $t-1$ within
a distance $R_j$ to its left, are activated. Let us call this set $A_t$.
If for some $t$, $A_t$ is empty the process stops. If the process never stops we say it survives.
We call the process \textit{homogeneous} if all $R_i$ have the same
distribution and $u_i=i$ for all $i$, otherwise we call it \textit{heterogeneous}.
We focus to the cases $\bbP(R_i < 1) \in (0,1)$ for all $i$. Unless stated
differently, we assume $u_i = i$ for all $i$.

Let $S$ be the event ``\textit{the reverse process survives}".

\subsubsection{The homogeneous case}
\label{S: horc}

\begin{teo}
\label{T: horfp}
Consider the \horfp. We have that
\begin{enumerate}
\item[\textit{(I)}] If $\mathbb{E}(R) = \infty$ then $\bbP(S) = 1.$
\item[\textit{(II)}] If $\mathbb{E}(R) < \infty$ then $\bbP(S) = 0.$
\end{enumerate}
\end{teo}

\begin{rem}
\label{T: Hetero Power Law}
For a random variable $R,$ having a power law distribution as in~(\ref{E: Power Law}), we have that

\begin{itemize}
\item if $1 < \alpha \leq 2$ then $\bbE[R] = \infty,$
\item if $\alpha > 2$ then $\bbE[R] < \infty.$
\end{itemize}

In conclusion, if $R$ has a power law distribution as in~(\ref{E: Power Law}),
with $\alpha=2$, then $\bbP[V] = 0$ for the \hofp\ by Remark~\ref{R: PowerLaw} and
$\bbP[S] = 1$ for the \horfp.

\end{rem}

\subsubsection{The heterogeneous case}
\begin{teo}
\label{T: PFRHe} Consider the \herfp. It holds that
\begin{enumerate}
\item[\textit{(I)}] $\sum_{k=1}^{\infty}\bbP(R_{n+k} \geq k) =
\infty$ for all $n$ if and only if $\bbP(S) = 1.$
\item[\textit{(II)}] If
$\sum_{n=1}^{\infty}\prod_{k=1}^{\infty}\bbP(R_{n+k} < k) <
\infty$ then $\bbP(S)> 0.$
\end{enumerate}
\end{teo}
%$\sum_{k=1}^{\infty}\bbP(R_{n+k} \geq k) = \infty$ for all $n$

\begin{rem}
Let $\rho = \sum_{n=1}^{\infty}\prod_{k=1}^{\infty}\bbP(R_{n+k} < k).$
Observe now that Theorem~\ref{T: PFRHe} gives no additional information for Theorem~\ref{T: horfp},
as in the homogeneous case $\rho$ equals either 0 ($\bbE[R]=\infty$) or $\infty$
($\bbE[R] < \infty$).
\end{rem}

\begin{rem}
\label{R: PFRHe}
By a coupling argument and Theorem~\ref{T: horfp} one can see that
if there is a random variable $R$, which distribution is $\mathbf{P}$, with
$\bbE[R] < \infty$ ($\bbE[R] = \infty$), such that
$\bbP(R_{n} \geq k) \leq \mathbf{P}(R \geq k)$
($\bbP(R_{n} \geq k) \geq \mathbf{P}(R \geq k)$) for all $k$
then $\bbP(S) = 0$ ($\bbP(S) = 1$).
\end{rem}

\section{Proofs}
\label{S: proofs}

Next we present some basic facts, starting from
the Raabe�s test (Fort~\cite[p.~32]{Fort} or Bonar and Khoury~\cite[p.~48]{BK}).

\begin{raabe}
For $a_n > 0,$ let us define
\begin{displaymath}
L = \lim_{n \rightarrow \infty} n \displaystyle \left (\frac{a_n}{a_{n+1}} - 1 \right).
\end{displaymath}
Then
\begin{itemize}
\item If $L > 1$ then $\sum_{n=1}^{\infty}a_n < \infty.$
\item If $L < 1$ then $\sum_{n=1}^{\infty}a_n = \infty.$
\item If $L = 1$ and $n \Big({{a_n}/{a_{n+1}}} - 1 \Big) \leq 1,$
for $n$ large enough, \\ then $\sum_{n=1}^{\infty}a_n = \infty.$
\end{itemize}
\end{raabe}

The following result (Bremaud~\cite[p.~422]{Bremaud}) is useful for what
comes next

\begin{lem}
Let $\{a_n\}_{n \geq 1}$ be a sequence of real numbers in $(0,1).$ Then,

\begin{equation}
\label{E: CD}
\prod_{i=0}^{\infty}(1-a_k) = 0 \Leftrightarrow \sum_{i=0}^{\infty} a_k = \infty.
\end{equation}
\end{lem}

\begin{rem}
\label{R: observe1}
Consider that the actionable vertices are at integer positions
$0=u_0 < u_1 < u_2 < \dots $ such that $u_{n+1} - u_n \leq m$, for $m \geq 1.$
From the definition of $V_n$ one can see that
\begin{itemize}
\item $V_{k+1} \supset V_{k}
\bigcap \left \{ \bigcup_{i=0}^k (R_{k-i} \geq (i+1)m ) \right \},$
\item $V_{k}$ e $\bigcup_{i=0}^k (R_{k-i} \geq (i+1)m)$ are increasing events,
\item $ \bbP (V_{n}) > 0$ for all $n.$
\end{itemize}
\end{rem}
From FKG inequality (Grimmett~\cite[p.34]{Grimmett}) we can assure that
\begin{eqnarray}
\label{E: uii}
\bbP(V_{k+1}) &\geq& \bbP(V_{k}
\cap \left \{ \cup_{i=0}^k (R_{k-i} \geq (i+1)m ) \right \}) \\
\nonumber & \geq & \bbP(V_{k})
\bbP(\left \{ \cup_{i=0}^k (R_{k-i} \geq (i+1)m ) \right \}) \\
\nonumber & = & \bbP(V_{k}) \Big[1-\prod_{i=0}^{k}\bbP(R_{k-i} < (i+1)m)\Big]
\end{eqnarray}
\noindent
and then
\begin{equation*}
\label{E: Vn}
\bbP(V_n) \geq
\prod_{j=0}^{n-1}\Big[1-\prod_{i=0}^{j}\bbP(R_{j-i} < (i+1)m)\Big].
\end{equation*}
Therefore
\begin{equation}
\label{E: V}
\bbP(V) \geq
\prod_{j=0}^{\infty}\Big[1-\prod_{i=0}^{j}\bbP(R_{j-i} < (i+1)m)\Big].
\end{equation}
Inequality~(\ref{E: uii}) becomes an equality if $u_i=mi$ for all $i \in \bbN$ and
some $m \in \bbN.$
From the latter set of displays and~(\ref{E: CD}) follows next
proposition.
\begin{prop}
\label{P: PFmae} Consider a \hefp\ which actionable vertices are
at integer positions $0=u_0 < u_1 < u_2 < \cdots $ such that $
u_{n+1}-u_n \leq m.$ Let $a_n = \prod_{i=0}^{n}\bbP(R_{n-i} <
(i+1)m)$ and assume $\bbP(R_{i} < m) \in (0,1).$
\begin{equation}
\label{E: PFmae}
\hbox{If } \sum_{n=0}^{\infty}a_n < \infty  \hbox{ then } \bbP[V]>0.
\end{equation}
\end {prop}

\noindent
\subsection{Firework Process}

\begin{proof}[Proof of Theorem~\ref{T: Criterio}]
Assume $\sum_{n=0}^{\infty} a_n < \infty.$ From
Proposition~\ref{P: PFmae}, with $m=1$, we have that $\bbP[V]>0.$

Assume now $\sum_{n=0}^{\infty} a_n = \infty.$ First consider the event
\[ C = \{ \exists n \hbox{ such that } \forall u_i>n \ \exists
x \hbox{ such that } x < u_i \leq x + R_x \}.\]

In words that means that from some position on, all vertex belong to the radius of
influence of some other vertices. Those later vertices not necessarily have been activated.
%Observe that $\bbP[V] \leq \bbP[C]$ as $V \subset C.$

Next, consider the following event
\[ B(u_n) = \{ u_n > x + R_x, \hbox{ for all } x < u_n \} \]

In words, the vertex $u_n$ does not belong to the radius of influence of any
vertex to its left.

Assuming all independent random variables having the same distribution as $R$ and that
$u_i=i$ ($B_n = B(u_n)$),
\begin{align*}
\bbP(B_n) = \bbP \left(\bigcap_{i=1}^{n}[R_{n-i} < i] \right) =
\prod_{i=1}^{n}\bbP(R < i) = a_{n-1}.
\end{align*}
Conditional independence of the $B_i$s as stated next, for $i>j$
\begin{eqnarray*}
\bbP(B_i \cap B_j) &=& \mathbb{P}\left(
\bigcap_{k=1}^{i-j}[R_{i-k} < k] \cap \bigcap_{k=1}^{j}[R_{j-k} < k] \right)\\
&=& \prod_{k=1}^{i-j}\mathbb{P}(R < k)\prod_{k=1}^{i}\mathbb{P}(R < k)\\
&=& \bbP(B_{i-j})\bbP(B_j)
\end{eqnarray*}
\noindent
makes the $B_i$s satisfy the definition of a renewal event in~\cite[p.308]{Fel}. So, from
the fact that $\sum_{n=1}^{\infty}\bbP(B_n) = \infty,$ one can rely on Theorem 2 of
Section XIII.3 of~\cite[p.312]{Fel} to see that

%\begin{align}
%\label{somainfinita}
%\sum_{n=1}^{\infty}\bbP(B_n) = \infty
%\end{align}
%\noindent
%are the two sufficient conditions for Theorem 2 of Feller~\cite[p.312]{Fel} to guarantee
%that
\[ \bbP(B_n \textrm { infinitely often }) = 1. \]
\noindent
From this we conclude that $\bbP[V] = 0 $, as
\[ V^c \supset C^c \supset \{B_n \textrm { infinitely often }\}. \]

Inequality~(\ref{PVinequalityGEQ}) follows from~(\ref{E: V}) and 
inequality~(\ref{PVinequalityLEQ}) follows from the fact that

\begin {align*}
V^C \supseteq \bigcup_{k=0}^{\infty}\displaystyle \left[R_0 = k,
\bigcap_{j=1}^{k}[R_j \leq k-j] \right].
\end {align*}

\end{proof}

\begin{proof}[Proof of Corollary~\ref{C: FPHo}]
Observe that, as $a_n=\prod_{i=0}^{n}\bbP(R < i+1)$
\[ \frac{a_n}{a_{n+1}}-1=\frac{\bbP(R \geq n+2)}{\bbP(R < n+2)}.\]
\noindent
Therefore
\begin{equation}
\label{E: Raabe}
\lim_{n \rightarrow \infty} n\Big(\frac{a_n}{a_{n+1}}-1\Big)=
\lim_{n \rightarrow \infty}n \bbP(R \geq n).
\end{equation}
\noindent From~(\ref{E: Raabe}), Raabes Test and Theorem~\ref{T:
Criterio}, follow ($I$), ($II$) and ($III$).
%($IV$) follows from~(\ref{E: V}).
\end{proof}

\begin{proof}[Proof of Corollary~\ref{C: lei de potencia}]
Observe that
\begin{align*}
\frac{1}{(\alpha - 1)(n+1)^{\alpha - 1}} & =
\int_{n+1}^{\infty}\frac{1}{x^{\alpha}}dx
\leq \sum_{j=n+1}^{\infty}\frac{1}{j^{\alpha}} \\ & \leq
\int_{n+1}^{\infty}\frac{1}{(x-1)^{\alpha}}dx
= \frac{1}{(\alpha -
1)n^{\alpha - 1}}.
\end{align*}
Then
\begin{displaymath}
\frac{Z_{\alpha}}{(\alpha - 1)}\frac{1}{(n+1)^{\alpha - 1}}
\leq \bbP(R \geq n) \leq \frac{Z_{\alpha}}{(\alpha - 1)}\frac{1}{n^{\alpha - 1}}.
\end{displaymath}
Consequently
\[
\lim_{n \rightarrow \infty}n\bbP(R \geq n) =
\left\{\begin{array}{ll}
+ \infty & \mbox{if $\alpha < 2,$} \\
{6 / \pi^2} & \mbox{if $\alpha = 2,$} \\
0 & \mbox{if  $\alpha > 2.$}
\end{array}
\right.
\]
From Corollary~\ref{C: FPHo}, the conclusion follows.
\end{proof}

\noindent \textit{Proof of Theorem~\ref{T: PFHeVive}}

Let

\[ a_n  =  \prod_{j=0}^{n}\bbP(R_{n-j} < (j+1)m ).\]

\noindent
\textit{Proof of (I).}
As

\[ \sum_{n=t}^{\infty}[\bbP(R_{n} < tm)]^t < \infty \]

\noindent
implies that

\[ \sum_{n=t}^{\infty}\Big[\max_{j \in \{0, \dots ,t-1 \}}\{
\bbP(R_{n-j} < t m )\}\Big]^t < \infty, \]

\noindent
and as

\[ a_n \leq \prod_{j=0}^{t-1}\bbP(R_{n-j} < tm) \leq \Big[\max_{j \in
\{0,\dots,t-1 \}}\{ \bbP(R_{n-j} < tm )\}\Big]^t, \]

\noindent
the series which terms are $a_n$ converges. So we can
use~(\ref{E: PFmae}) in order to get the result.

\noindent
\textit{Proof of (II).}
Let

\[ r_n = \prod_{j=0}^{n} [\mathbf{P}(R < (j+1)m ) + b_{(j+1)m}]. \]

As

\begin{displaymath}
n \Big(\frac{r_n}{r_{n+1}} - 1 \Big) = \frac{n[\mathbf{P}(R \geq (n+2)m ) -
b_{(n+2)m} ]}{\mathbf{P}(R < (n+2)m ) +  b_{(n+2)m}},
\end{displaymath}

from the hypothesis
\begin{displaymath}
\lim_{n \rightarrow \infty}n \Big(\frac{r_n}{r_{n+1}}-1 \Big) > 1.
\end{displaymath}

But $a_n \leq r_n$, therefore the series which terms are $a_n$ is convergent and so
we can use Proposition~\ref{P: PFmae} to get the desired result.

\noindent
\textit{Proof of (III).} It follows from~(\ref{E: V}).
\qed

\subsection{Reverse Firework Process}

First consider the following variation of the \hofp. Instead of
having just the origin activated at time zero, we consider that
all vertices to its left are also activated at time zero. The set
of independent random variables which defines the radius of
influence of all vertices is $\{F_i\}_{i\in\bbZ},$ all ha\-ving
the same distribution as $R$, the random variable which defines
the \horfp.

For this variation of the \hofp\ let us define the following events
\begin{itemize}
\item $\mathcal{V}_n = $ \textit{the vertex $n$ is hit by an explosion},
\item $\mathcal{V} = $ \textit{the process survives}.
\end{itemize}
By analogy ``\textit{to survive}'' in this variation means to hit
infinitely many vertices of $\bbN.$ It follows that
\begin{equation}
\label{E: Vbonito}
\mathcal{V} =
\bigcap_{n=0}^{\infty}\bigcup_{j=0}^{\infty}[F_{n-j}
\geq j+1].
\end{equation}

\begin{prop}
\label{P:Morter}
If
$\mathbb{E}(R) < \infty$, then $\bbP(\mathcal{V})=0$.
\end{prop}

\noindent \textit{Proof of Proposition \ref{P:Morter}} Let us define
the following events
\begin{displaymath}
\mathcal{A}_n = \bigcup_{i=-\infty}^{n-1}\{F_{i} \geq
2n-i\}
\end{displaymath}
\noindent
and
\begin{displaymath}
\mathcal{B}_n = \bigcup_{i=n}^{2n-1}\{F_{i} \geq 2n-i\}.
\end{displaymath}

Observe that
\begin{displaymath}
\mathcal{V}_{2n} \subseteq \mathcal{V}_n \cap [\mathcal{A}_n \cup
\mathcal{B}_n].
\end {displaymath}

Therefore
\begin{equation*}
\label{relan2nr}
\bbP(\mathcal{V}_{2n})
\leq \bbP(\mathcal{A}_n) +
\bbP(\mathcal{B}_n)\bbP(\mathcal{V}_n).
\end{equation*}

Now
\begin{align*}
\bbP(\mathcal{A}_n) &\leq \sum_{i=-\infty}^{n-1}\bbP(F_{i} \geq 2n-i) \\
&=\sum_{i=n+1}^{\infty}\bbP(F_{2n-i} \geq i) \\
&=\sum_{i=n+1}^{\infty}\bbP(R \geq i) \longrightarrow 0,
\end{align*}
and
\begin{align*}
\bbP(\mathcal{B}_n) &= \bbP\displaystyle\left(\bigcup_{i=n}^{2n-1}
\{F_i \geq 2n-i \}\right) \\
&= 1 - \prod_{i=n}^{2n-1}\bbP(F_i < 2n - i) \\
&\leq 1 - \prod_{i=1}^{\infty}\bbP(R < i).
\end{align*}

\noindent
Then,~(\ref{E: CD}) and $\mathbb{E}(R) < \infty$
guarantee the existence of $\lambda$ $\in$ (0,1) such that
\begin {displaymath}
\bbP(\mathcal{B}_n)\leq \lambda
\end {displaymath}
for all $n.$  So, as for the homogeneous case
$\bbP(\mathcal{A}_n) \geq \bbP(\mathcal{A}_{n+1})$,
\begin{align*}
\label{E: deflema} \lim_{n \rightarrow
\infty}\bbP(\mathcal{V}_{n}) = 0
\end{align*}
and this implies that $\bbP(\mathcal{V}) = 0$ as
$\mathcal{V}_{n+1} \subset \mathcal{V}_{n}$. \qed

%\begin{align*}
%\bbP(\mathcal{V}_{2n}) \leq \bbP(\mathcal{A}_{n}) +
%\lambda \bbP(\mathcal{V}_{n}) \textrm { for all } n,
%\end{align*}
%which implies that
%\begin{align*}
%\bbP(\mathcal{V}_{4^{m}}) \leq \frac{\displaystyle \left(1-
%\lambda^{m}\right)}{(1 - \lambda)}\bbP(\mathcal{A}_{2^{m}})
%+ \lambda^m.
%\end{align*}
%
%So
%\begin{align*}
%\label{E: deflema}
%\lim_{m \rightarrow \infty}\bbP(\mathcal{V}_{4^m}) = 0,
%\end{align*}
%and this implies that $\bbP(\mathcal{V}) = 0$ as
%$\mathcal{V}_{n+1} \subset \mathcal{V}_{n}$. \qed

\noindent \textit{Proof of Theorem~\ref{T: horfp}}

Let $\{R_i\}_{i \in \bbN}$ independent random variables
distributed as $R.$  Observe that
\begin{equation}
\label{E: Sbonito}
S = \bigcap_{n=0}^{\infty}\bigcup_{j=1}^{\infty}[R_{n+j} \geq j].
\end{equation}

By using FKG inequality (Grimmett~\cite[p.34]{Grimmett}) and the fact that intersections of increasing events
is an increasing event, we have that

\begin{align*}
\mathbb{P}\displaystyle
\left(\bigcap_{n=0}^{n_0}\bigcup_{j=1}^{\infty}[R_{n+j} \geq
j]\right) \geq \prod_{n=0}^{n_0}\mathbb{P}\displaystyle
\left(\bigcup_{j=1}^{\infty}[ R_{n+j} \geq j] \right)
\end{align*}

for all $n_0.$ Taking the limit $n_0
\rightarrow \infty$, by the continuity of probability

\begin{align*}
\mathbb{P}\displaystyle
\left(\bigcap_{n=0}^{\infty}\bigcup_{j=1}^{\infty}[R_{n+j} \geq
j]\right) \geq \prod_{n=0}^{\infty}\mathbb{P}\displaystyle
\left(\bigcup_{j=1}^{\infty}[ R_{n+j} \geq j] \right).
\end{align*}

Therefore

\begin{equation}
\label{l1}
\bbP(S) \geq
\prod_{n=0}^{\infty}\Big[1-\prod_{j=1}^{\infty}[1-\bbP(R_{n+j}
\geq j)]\Big].
\end{equation}

\noindent
\textit{Proof of (I).} From the hypothesis

\begin{equation}
\label{E: diverge}
\sum_{j=1}^{\infty}\bbP(R \geq j) = \infty.
\end{equation}

Now, (\ref{E: CD}) and~(\ref{E: diverge}) implies that

\[ \prod_{j=1}^{\infty}[1 - \bbP(R \geq j)] = 0,\]

\noindent
and $\bbP(S) = 1$ follows by~(\ref{l1}).

\noindent
\textit{Proof of (II).}
%First remember that
%$R_i, F_i$ and $R$ have the same distribuition.
By Proposition~\ref{P:Morter}, (\ref{E: Vbonito}) and the
fact that $R_i$ and $F_i$ have the same distribution

\begin{equation}
\label{E: key}
\bbP
\left(\bigcap_{n=0}^{\infty}\bigcup_{j=0}^{\infty}[R_{n-j}
\geq j+1]\right) = 0.
\end{equation}

\noindent
By the other hand, as $R_i$ are all distributed as $R$
\begin{align*}
\bbP \left(\bigcap_{n=0}^{\infty}\bigcup_{j=0}^{\infty}[R_{n-j}
\geq j+1]\right) = \bbP\left(\bigcap_{n=0}^{\infty}\bigcup_{j=0}^{\infty}[R_{n+j+1} \geq
j+1]\right),
\end{align*}

\noindent
and therefore, by~(\ref{E: Sbonito}) and~(\ref{E: key}), $\bbP(S)=0.$ \qed

\noindent \textit{Proof of Theorem~\ref{T: PFRHe}}

\noindent \textit{Proof of (I).} Assuming that
$\sum_{k=1}^{\infty}\bbP(R_{n+k} \geq k) = \infty $ for all $n$
and considering~(\ref{E: CD}), one can see that
$\prod_{k=1}^{\infty}[1-\bbP(R_{n+k} \geq k)] = 0$ for all $n$.
Therefore, by~(\ref{l1}), $\bbP(S)=1.$ By the other side, as
$\bbP(S)  \leq 1 - \prod_{k=1}^{\infty}\bbP(R_{n+k} <k)$ for all
$n$, if $\bbP(S)=1$ we have that
$\prod_{k=1}^{\infty}[1-\bbP(R_{n+k} \geq k)] = 0$ for all $n$.
Now, from~(\ref{E: CD}), $\sum_{k=1}^{\infty}\bbP(R_{n+k} \geq k)
= \infty$ for all $n$.

\noindent
\textit{Proof of (II).}
From $\sum_{n=1}^{\infty}\prod_{k=1}^{\infty}\bbP(R_{n+k} < k) <
\infty$, follows that, by the use of~(\ref{E: CD}),
$\prod_{n=0}^{\infty}[1-\prod_{k=1}^{\infty}[1-\bbP(R_{n+k}
\geq k)]]>0$. Then, by~(\ref{l1}) we have that $\bbP(S)>0.$
\qed

\section{Final Remarks and Examples}
\label{S: FR}

We consider two discrete propagation phenomena modeling in their homogeneous and heterogeneous versions.
While the \textit{Firework Process} models a phenomena where there is at all times a finite number of 
individuals trying to spread an information for an infinite group of individuals, the 
\textit{Reverse Firework Process} models a phenomena where there is always an infinite number of individuals 
willing and working towards to heard about that information from a finite 
quantity of informed individuals. Our results show that the four versions are qualitatively different.

Considering the \hofp, Remark~\ref{R: ERfiniteDyes} shows that the information will no be spread for an infinite number of
individuals if $\bbE[R]$ is finite. To have a \textit{radius of influence} with infinite expectation 
is also no guarantee for the information to reach an infinite number of individuals, as Remark~\ref{R: PowerLaw} shows. 
Besides, the probability of not reaching an infinite amount 
of individuals is at least $\bbP[R=0]$. Conversely, in the \hefp, to have an infinite expectation guarantees 
almost surely that the information will spread out among an infinite amount of individuals, as Theorem~\ref{T: horfp}
points out. Furthermore in the case where the \textit{radius of influence} has a power 
law distribution as in (\ref{E: Power Law}), the process works in opposite direction as Remark~\ref{T: Hetero Power Law} 
shows for $\alpha=2$. The processes agree for R whose expectation is finite. Next we present some final examples pointing to some
extreme cases.

Let $\{b_n\}_{n \in \bbN}$ be a non-increasing sequence convergent
to 0 and such that $b_0<1.$

\begin{exa}
\label{E: um}

It is possible to have in the \hefp\  the expectation of the radius of
influence infinite for all vertices together and the process dies out almost surely.

\begin{enumerate}
\item[\textit{(i)}] $\bbP(R_n = 0) = 1 - b_n$ and $\bbP(R_n = k) = b_{n+k-1} - b_{n+k}$ for $k \geq 1.$
\item[\textit{(ii)}] $\sum_{n=0}^{\infty}b_n = \infty.$
\item[\textit{(iii)}] $\lim_{n \rightarrow \infty} n b_n = 0.$
\end{enumerate}

Observe that $\mathbb{E}(R_n) = \infty$ for all $n$ from \textit{(ii)}. Besides
$\bbP[V]=0$ from \textit{(iii)}, because

\begin{equation}
\label{E: example}
\bbP(V_n) \leq \sum_{k=0}^{n-1}\bbP(R_k \geq n-k) =
\sum_{k=0}^{n-1}b_{n-1} = (n-1)b_n.
\end{equation}
\end{exa}

\begin{exa}

It is possible to have in the \hefp\ the expectation of the radius of
influence finite for all vertices and the process survives with positive probability.
Assume that $\sum_{n=0}^{\infty}b_n < \infty,$ while

\begin{itemize}
\item $\bbP(R_n = 0) = b_n$
\item $\bbP(R_n = 1) =1 - b_n$
\end{itemize}

Then $\mathbb{E}(R_n) < 1$ for all $n$ and
$\mathbb{P}(V) > 0$ by part ($I$) of Theorem~\ref{T: PFHeVive}
with $m=t=1$.
\end{exa}

\begin{exa} Next we present an example where $\bbP[S]=1$ for
a \herfp\ while $\bbP[V]=0$ for a \hefp. For this aim consider

\begin{enumerate}
\item[\textit{(i)}] $\bbP(R_n = 0) = 1 - b_n$ and $\bbP(R_n = n) = b_n.$
\item[\textit{(ii)}] $\sum_{n=0}^{\infty}b_n = \infty.$
\item[\textit{(iii)}] $\lim_{n \rightarrow \infty} n b_n = 0.$
\end{enumerate}

Observe that even though $\lim_{n \to \infty} \bbE[R_n] = 0$ and
$\lim_{n \to \infty} \bbP(R_n = 0) = 1,$
from Theorem~\ref{T: PFRHe} and \textit{(ii)} it is true for the \herfp\
that $\mathbb{P}(S) = 1.$ In the opposite direction, by~(\ref{E: example})
and \textit{(iii)} one have that $\bbP[V]=0$ for the \hefp.

\end{exa}


\begin{thebibliography}{99}

\bibitem{Roy}{S.Athreya, R.Roy and A.Sarkar}.
On the coverage of space by random sets.
\textit{Advances in Applied Probability}.
Volume 36, Number 1, 1-18 (2004).

\bibitem{BZ} {D.Bertacchi and F.Zucca}.
Critical behaviors and critical values of branching random walks on multigraphs.
\textit{J. Appl. Probab.} \textbf{45}, 481--497 (2008).

\bibitem{BK}{D.D.Bonar and M.J.Khoury Jr}.
Real infinite series,
Mathematical Association of America Textbooks (2006).

\bibitem{Bremaud}{P.Bremaud}.
\textit {Markov chains. Gibbs fields, Monte Carlo simulation, and queues.}
Texts in Applied Mathematics, 31. Springer-Verlag, New York (1999).

\bibitem{DK}{D.Daley and D.G.Kendall}.
Stochastic rumours.
\textit{J Inst Math Appl}
1, 42 - 55 (1965).

\bibitem{Fel}{W.Feller}.
An Introduction to Probability Theory and its
Applications, Vol 1, 3rd ed, John Wiley, New York (1968).

\bibitem{Fort}{T.Fort}.
Infinite Series, Oxford (1930).

\bibitem{Grimmett}{G.Grimmett}.
Percolation,
(2nd ed.), Springer-Verlag, New York (1999).

\bibitem{Isham}{V.Isham, S.Harden and M.Nekovee}.
Stochastic epidemics and rumours on finite random networks.
\textit{Physica A: Statistical Mechanics and its Applications}
Volume 389, Issue 3, Pages 561-576 (2010).

\bibitem{Kesten}{H.Kesten and V.Sidoravicius}.
The spread of a rumor or infection in a moving population.
\textit{Ann. Probab.}
33, no. 6, 2402--2462 (2005).

\bibitem{KOW}{J.Kostka, Y.A.Oswald and R.Wattenhofer}.
Word of Mouth: Rumor Dissemination in Social Networks
in \textit{Structural Information and Communication Complexity}
15th International Colloquium, SIROCCO 2008 Villars-sur-Ollon,
Switzerland, June 17-20, Proceedings (2008).

\bibitem{Kurtz}{T.G.Kurtz, E.Lebensztayn, A.Leichsenring and F.P.Machado}.
Limit theorems for an epidemic model on the complete graph.
\textit{Alea}(Online), v. IV, p. 3 (2008).

\bibitem{LMM}{E.Lebensztayn, F.P.Machado, M. Z.Martinez}.
Self avoiding random walks on homogeneous trees.
Markov Processes and Related Fields, v. 12, p. 735-745 (2006).

\bibitem{LR}{E.Lebensztayn, P.M.Rodriguez}.
The disk-percolation model on graphs.
Statistics $\&$ Probability Letters,
vol. 78, issue 14, 2130-2136 (2008).

\bibitem{MT}{D.P.Maki and M.Thompson}.
Mathematical Models and Applications
Prentice-Hall, Englewood Cliffs.N.J. (1973).

\bibitem{Pearce}{C.E.M.Pearce}.
The Exact Solution of the General Stochastic Rumours
\textit{Mathematical and Computer Modelling},
31, 289-298 (2000).

\bibitem{Pittel}{B.Pittel}.
On a Daley--Kendall model of random rumours.
\textit{J. Appl. Prob.}
27, 14--27 (1990).

\bibitem{Sudbury}{A.Sudbury}.
The proportion of the population never hearing a rumour
\textit{J. Appl. Probab.}
22, 443-446 (1985).

\end{thebibliography}
\end{document}